\documentclass[11pt, oneside]{amsart}

\usepackage{amsmath, amsthm, amssymb, array, amsfonts, booktabs, wasysym, verbatim, bbm, color, graphics, geometry, physics, hyperref, tikz-cd,  mathrsfs, tensor, tabularx, float, mathtools, enumitem, xltabular}

\usepackage[utf8]{inputenc}
\usepackage{amsmath}
\usepackage[
backend=biber,
style=alphabetic,
sorting=nyt
]{biblatex}
\DeclareFontFamily{U}{min}{}
\DeclareFontShape{U}{min}{m}{n}{<-> udmj30}{}

\newsavebox{\pullback}
\sbox\pullback{%
\begin{tikzpicture}%
\draw (0,0) -- (1ex,0ex);%
\draw (1ex,0ex) -- (1ex,1ex);%
\end{tikzpicture}}

\DeclareMathOperator{\GL}{GL}
\DeclareMathOperator{\BGL}{BGL}

\newcommand{\Q}{\mathbb{Q}}

\newcommand{\Pp}{\mathbb{P}}

\newcommand{\Hom}{\mathrm{Hom}}

\newcommand{\Spec}{\operatorname{Spec}}

\newcommand{\codim}{\mathrm{codim}\,}

\newtheorem{thm}{Theorem}[section]
\newtheorem{prop}[thm]{Proposition}
\newtheorem{conv}[thm]{Convention}
\newtheorem{notn}[thm]{Notation}
\newtheorem{lem}[thm]{Lemma}
\newtheorem{cor}[thm]{Corollary}

\theoremstyle{definition}

\newtheorem{ex}[thm]{Example}
\newtheorem{rem}[thm]{Remark}

\title[Universal Formulas for de Jonqui\`eres Loci and Equivariant Contact Problems]{Universal Formulas for de Jonqui\`eres Loci\\ and Equivariant Contact Problems}
\author{Smita Rajan}

\begin{document}

\maketitle
\begin{abstract} 
Given a family of genus $g$ curves $f:\Psi\to S$ and a relative degree $d$ line bundle $\mathcal{L}$ on $\Psi$ such that $f_*\mathcal{L}$ is a vector bundle, the projectivization $\Pp(f_*\mathcal{L})$ is stratified by \textit{relative de Jonqui\`eres loci}, which parameterize sections of $\mathcal{L}$ with prescribed vanishing orders. We give universal formulas for the classes of these loci in terms of the relative $\mathcal{O}(1)$ on $\Pp(f_*\mathcal{L})$ and tautological classes pulled back from the universal Picard stack $\mathrm{Pic}_g^d$. When $f$ is the family of lines in $\Pp^r$, we use the universal polynomials to compute $\GL_{r+1}$-equivariant classes of loci of hypersurfaces and complete intersections admitting lines with prescribed contact orders.
\end{abstract}
\section{Introduction}
Given a smooth curve $C$ of genus $g$, and a linear series $(\mathcal{L},V)$ of degree $d$ and dimension $r$, one can compute the locus of all divisors $D=\sum_{i=1}^m \mu_ip_i$ such that  $\mathcal{L}\cong \mathcal{O}_C(D).$  Such divisors are called \textit{de Jonqui\`eres divisors of length $m$}. Embedding $C$ into $\Pp^{r}$ via $\mathcal{L}$, one sees that the locus of de Jonqui\`eres divisors of length $m$ is equivalent to the locus of hyperplanes intersecting $C$ at $m$ points with multiplicities $\mu=(\mu_1,\dots,\mu_m)$. In the case where the number of such divisors is finite, this classical question has a beautiful answer given by the de Jonqui\`eres formula \cite{arbarello1985geometry}: 
Let $m_a$ be the multiplicity of $a$ in the partition $\mu$. The number of de Jonqui\`eres divisors of length $m$ is the coefficient of the monomial $t_1\cdots t_m$ in 
\[\frac{1}{\prod_{a}m_a!}(1+\mu_1^2t_1+\dots+\mu_m^2t_m)^g(1+\mu_1t_1+\dots+\mu_mt_m)^{d-r-g}.\]
This paper studies the generalization of this problem to the case where $(C,\mathcal{L})$ varies in families. Suppose we have a family of  genus $g$ curves $f:\Psi\to S$ and a relative degree $d$ line bundle $\mathcal{L}$ on $\Psi$, such that $f_*\mathcal{L}$ is a vector bundle. For $\mu=(\mu_1,\dots,\mu_m)$ a partition of $d$, define 
the locally closed subsets
\[X_\mu^\circ=\left\{(C,s):\mathrm{div}(s)= \sum_{i=1}^m \mu_ip_i\right\}\subseteq \Pp(f_*\mathcal{L}),\]
which give rise to a natural stratification of $\Pp(f_*\mathcal{L})$. 
We call their closures $X_{\mu}=\overline{X_{\mu}^\circ}$ the \textit{relative de Jonqui\`eres loci of type $\mu$}.

Many well-studied loci are examples of de Jonqui\`eres loci. For example, when $f:\mathcal{C}\to \mathcal{M}_g$ is the universal curve and $\mathcal{L}=\omega_{f}^{\otimes k}$, the relative de Jonqui\`eres loci are (projectivized, unordered) strata of $k$-differentials, which stratify the $k$-Hodge bundle, $\Pp(f_*\omega_{f}^{\otimes k})=\Pp(\mathcal{H}^k).$ The fundamental classes of strata of $k$-differentials were computed in \cite{ChenCycle}, \cite{Sauvaget2019}, and also over the moduli space of curves with rational tails in \cite{GTIncidence}.
Additionally, when $f:\Pp\mathcal{W}\to \mathrm{BGL}_2$ is the projectivization of the universal rank $2$ vector bundle and $\mathcal{L}=\mathcal{O}_{\Pp\mathcal{W}}(d),$ the relative de Jonqui\`eres loci are the $\GL_2$-equivariant coincident root loci of binary forms. Their fundamental classes were computed in \cite{feher2003coincidentrootlocibinary} and $\mathrm{PGL}_2$-equivariantly in \cite{spink2018pgl2equivariantstratapointconfigurations}.
In the case that $S=\Spec k$, when the relative de Jonqui\`eres loci $X_{\mu}$ have codimension $r$ in $\Pp(H^0(C,\mathcal{L}))$, the degree of $X_{\mu}$ is equal to the number of de Jonqui\`eres divisors of length $m$ for a general $r$-dimensional linear subsystem $(\mathcal{L},V).$

In this paper, we provide universal formulas for relative de Jonqui\`eres loci that, under specialization, recover those above. These formulas allow both the curve and the line bundle to vary in moduli. The pair $(f:\Psi\to S,\mathcal{L})$ determines a morphism from $S$ to the universal Picard stack $\mathrm{Pic}_g^d.$ The universal curve $f:\mathcal{C}_g\to \mathrm{Pic}_g^d$ is equipped with the universal degree $d$ line bundle $\mathscr{L}$. From $c_1(\omega_{f})$ and $c_1(\mathscr{L})$, we can define tautological classes in $A^*(\mathrm{Pic}_g^d)$ called the \textit{twisted kappa classes}:
\[\kappa_{i,j}=f_*(c_1(\omega_f)^{i+1}c_1(\mathscr{L})^j).\]
In the following theorem, we give universal formulas for relative de Jonqui\`eres loci.
\begin{notn}
    We write $\mu=(\mu_1,\dots,\mu_m)$ and $\nu=(\nu_1,\dots,\nu_n)$ for tuples of positive integers, and set 
    \[c(\mu)=\sum_{i=1}^m \mu_i, \qquad c(\nu)=\sum_{i=1}^n \nu_i.\] In the setting of Theorems \ref{thm1} and \ref{thm2}, $\mu$ is a partition of $d$, and we use $\nu$ when 
    \[\mu=\left(\nu_1,\dots,\nu_n,1^{d-c(\nu)}\right), \qquad \nu_{i}>1.\]
\end{notn}
\begin{thm}\label{thm1}
    Let $(f:\Psi\to S,\mathcal{L})$ be a pair of a family of genus $g$ curves and a relative degree $d$ line bundle, and let $\alpha$ be the corresponding map $S\to \mathrm{Pic}_g^d.$ Suppose that $f_*\mathcal{L}$ is a vector bundle, and let
    $\gamma=c_1(\mathcal{O}(1))$ on $\Pp(f_*\mathcal{L})$. Suppose that $X_{\mu}$ has codimension $c(\nu)-n$, and that $\dim X_{\mu'}<\dim X_{\mu}$ for $\mu'$ a coarsening of $\mu.$
    Then, there exist explicit polynomials $f_\nu^i\in A^{i-n}(\mathrm{Pic}_g^d)$ in the twisted kappa classes such that the class of the relative de Jonqui\`eres locus $X_{\mu}$ is
    \[[{X}_\mu]=\frac{1}{\prod_{a} m_a!}\sum_{i=n}^{c(\nu)}\gamma^{c(\nu)-i} g_{\nu}^i,\]
    where $g_{\nu}^i\coloneq \alpha^*f_{\nu}^i,$ and $m_a$ is the multiplicity of $a$ in $\nu.$
\end{thm}
The key ingredient in our construction is a vector bundle $\mathcal{F}_{\nu}(\mathscr{L})$ on $\mathcal{C}_g^n\to \mathrm{Pic}_{g}^d$ which we call the \textit{weighted bundle of principal parts}. The vector bundle $\mathcal{F}_{\nu}(\mathscr{L})$ generalizes relative bundles of principal parts \cite{eisenbud20163264} and secant bundles on symmetric products \cite{Mattuck1965Secant} to the weighted setting. The fiber of $\mathcal{F}_{\nu}(\mathscr{L})$ over a geometric point $(C,\mathcal{L},p_1,\dots,p_n)$ is 
\[H^0(C,\mathcal{L}\otimes \mathcal{O}_{C}/\mathcal{I}_{p_1}^{\nu_1}\dots \mathcal{I}_{p_n}^{\nu_n}).\]
The vector bundle $\mathcal{F}_{\nu}(\mathscr{L})$ has rank $c(\nu)$ and its construction commutes with base change. We use $\mathcal{F}_{\nu}(\mathscr{L})$ to pick up higher-order zeros of sections of $\mathcal{L}$. The advantage of working with $\mathcal{F}_{\nu}(\mathscr{L})$ compared to a direct sum of pullbacks of relative bundles of principal parts is that the diagonals do not have to be considered separately; when the points $p_i,p_j$ collide, the multiplicity at $p_i$ in the fiber is $\nu_i+\nu_j$ rather than $\max(\nu_i,\nu_j).$ 

The bundle $\mathcal{F}_{\nu}(\mathscr{L})$ admits a filtration by line bundles, and its  Chern classes are sums of products of pullbacks of $c_1(\mathscr{L})$ and $c_1(\omega_{f})$ with classes of diagonals $\Delta_{ij}\subseteq \mathcal{C}_g^n,$ allowing us to calculate them explicitly. The $f_{\nu}^i$ in Theorem \ref{thm1} are pushforwards of the Chern classes of $\mathcal{F}_{\nu}(\mathscr{L})$ under $g:\mathcal{C}_g^n\to \mathrm{Pic}_g^d$. 
    \begin{ex}\label{intro example}
    Let $d=6$ and $\mu=\nu=(2,2,2)$. The universal polynomials of Theorem \ref{thm1} are as follows:
\begin{align*}
f_{\nu}^{3}={}&
    8\kappa_{-1,1}^{3}
    +12\kappa_{-1,1}^{2}\kappa_{0,0}
    +6\kappa_{-1,1}\kappa_{0,0}^{2}
    +\kappa_{0,0}^{3}
    -72\kappa_{-1,1}^{2} \\& \quad
    -96\kappa_{-1,1}\kappa_{0,0}
    -30\kappa_{0,0}^{2}
    +160\kappa_{-1,1}
    +176\kappa_{0,0},
\\[0.75em]
f_{\nu}^{4}={}&
    12\kappa_{-1,1}^{2}\kappa_{-1,2}
    +12\kappa_{-1,1}\kappa_{-1,2}\kappa_{0,0}
    +3\kappa_{-1,2}\kappa_{0,0}^{2}
    +12\kappa_{-1,1}^{2}\kappa_{0,1} \\
&\quad
    +12\kappa_{-1,1}\kappa_{0,0}\kappa_{0,1}
    +3\kappa_{0,0}^{2}\kappa_{0,1}
    -108\kappa_{-1,1}\kappa_{-1,2}
    -66\kappa_{-1,2}\kappa_{0,0} \\
&\quad
    -156\kappa_{-1,1}\kappa_{0,1}
    -90\kappa_{0,0}\kappa_{0,1}
    -36\kappa_{-1,1}\kappa_{1,0}
    -18\kappa_{0,0}\kappa_{1,0} \\
&\quad
    +240\kappa_{-1,2}
    +528\kappa_{0,1}
    +256\kappa_{1,0},
\\[0.75em]
f_{\nu}^{5}={}&
    6\kappa_{-1,1}\kappa_{-1,2}^{2}
    +3\kappa_{-1,2}^{2}\kappa_{0,0}
    +12\kappa_{-1,1}\kappa_{-1,2}\kappa_{0,1}
    +6\kappa_{-1,2}\kappa_{0,0}\kappa_{0,1} \\
&\quad
    +6\kappa_{-1,1}\kappa_{0,1}^{2}
    +3\kappa_{0,0}\kappa_{0,1}^{2}
    -36\kappa_{-1,2}^{2}
    -24\kappa_{-1,1}\kappa_{-1,3}
    -12\kappa_{-1,3}\kappa_{0,0} \\
&\quad
    -96\kappa_{-1,2}\kappa_{0,1}
    -60\kappa_{0,1}^{2}
    -60\kappa_{-1,1}\kappa_{0,2}
    -30\kappa_{0,0}\kappa_{0,2}
    -18\kappa_{-1,2}\kappa_{1,0} \\
&\quad
    -18\kappa_{0,1}\kappa_{1,0}
    -36\kappa_{-1,1}\kappa_{1,1}
    -18\kappa_{0,0}\kappa_{1,1}
    +160\kappa_{-1,3}
    +528\kappa_{0,2} \\
&\quad
    +512\kappa_{1,1}
    +120\kappa_{2,0},
\\[0.75em]
f_{\nu}^{6}={}&
    \kappa_{-1,2}^{3}
    +3\kappa_{-1,2}^{2}\kappa_{0,1}
    +3\kappa_{-1,2}\kappa_{0,1}^{2}
    +\kappa_{0,1}^{3}
    -12\kappa_{-1,2}\kappa_{-1,3} \\
&\quad
    -12\kappa_{-1,3}\kappa_{0,1}
    -30\kappa_{-1,2}\kappa_{0,2}
    -30\kappa_{0,1}\kappa_{0,2}
    -18\kappa_{-1,2}\kappa_{1,1}
    -18\kappa_{0,1}\kappa_{1,1} \\
&\quad
    +40\kappa_{-1,4}
    +176\kappa_{0,3}
    +256\kappa_{1,2}
    +120\kappa_{2,1}.
\end{align*}
Taking $f:\mathcal{C}\to \mathcal{M}_g$, and $\mathcal{L}=\omega_{f}$ gives us $\alpha^*\kappa_{i,j}=\kappa_{i+j}.$ The fundamental class of the stratum of differentials for this partition when $k=1$ and $g=4$ in $A^3(\Pp \mathcal{H}_4)$ is
\[
\begin{aligned}
[X_{\mu}]=\frac{1}{6}\Bigl[{}&
\gamma^3\Bigl(
27\kappa_0^3-198\kappa_0^2+336\kappa_0
\Bigr)
+\gamma^2\Bigl(
54\kappa_0^2\kappa_1-474\kappa_0\kappa_1+1024\kappa_1
\Bigr) \\
&\quad
+\gamma\Bigl(
36\kappa_0\kappa_1^2-228\kappa_1^2
-180\kappa_0\kappa_2+1320\kappa_2
\Bigr)
+8\kappa_1^3-120\kappa_1\kappa_2+592\kappa_3\Bigr].
\end{aligned}
\]
Taking $f:\Pp\mathcal{W}\to \mathrm{BGL}_2$ and $\mathcal{L}=\mathcal{O}_{\Pp\mathcal{W}}(d)$ gives us $\alpha^*\kappa_{i,j}=f_*((-2\zeta-c_1)^{i+1}(d\zeta)^j),$ where $\zeta=c_1(\mathcal{O}_{\Pp\mathcal{W}}(1)).$
It follows that the $\GL_2$-equivariant class of the coincident root locus in $A^3(\Pp(\mathrm{Sym}^6 \mathcal{W}^\vee))$ is
\[\begin{aligned}
   [X_{\mu}]=\frac{1}{6}\Bigl(48\gamma^3-432c_1\gamma^2+(1104c_1^2 + 768c_2)\gamma-720c_1^3 - 2304c_1c_2\Bigr).
\end{aligned}\]
These formulas agree with the known formulas in \cite{ChenCycle} and \cite{feher2003coincidentrootlocibinary}.
\end{ex}
In the case that $f:\Pp \mathcal{S}\to \mathrm{Gr}(2,r+1)$ is the family of lines in $\Pp^r$ and $\mathcal{L}=\mathcal{O}_{\Pp \mathcal{S}}(d)$, we use the universal polynomials $f_{\nu}^i$ to study equivariant contact problems. Let $X$ be a degree $d$ hypersurface in $\Pp^r$. 
For $\mu$ a partition of $d$, we define a $\mu$-incident line to be a line meeting $X$ at points with multiplicities prescribed by $\mu'$, for $\mu'=\mu$ or $\mu'$ a partition that is a coarsening of $\mu$. Let $\mathcal{W}$ be the universal rank $r+1$ vector bundle over $\mathrm{BGL}_{r+1}$. We consider the universal $\mu$-incident line $\Sigma_{\mu}$, which is equipped with projection maps to the moduli space of degree $d$ hypersurfaces $\Pp(\mathrm{Sym}^d\mathcal{W}^\vee)$ and $\mathrm{Gr}(2,\mathcal{W}):$
\begin{center}
\begin{tikzcd}[row sep=50]
\Sigma_{\mu} \arrow[r, hook] & {\mathrm{Gr}(2,\mathcal{W})\times_{\BGL_{r+1}}\Pp(\mathrm{Sym}^d\mathcal{W}^\vee)} \arrow[d, "\pi_2"'] \arrow[r, "\pi_1"] & {\Pp(\mathrm{Sym}^d\mathcal{W}^\vee)} \arrow[d, "\varphi_2"] \\
                             & {\mathrm{Gr}(2,\mathcal{W})} \arrow[r, "\varphi_1"']                                                                                            & \mathrm{BGL}_{r+1}                                     
\end{tikzcd}
\end{center}
\begin{rem}
 In the special case that $r=1,$ $\mathrm{Gr}(2,2)=\mathrm{pt},$ so $[\mathrm{Gr}(2,2)/\GL_{2}]=\mathrm{BGL}_2$. Then, $[\Sigma_{\mu}]$ is the $\GL_2$-equivariant coincident root locus $[X_{\mu}]$.
\end{rem}
Let $d(\mu)$ denote the expected fiber dimension of $\pi_1|_{\Sigma_{\mu}}:$
\[d(\mu)=2(r-1)+n-c(\nu).\]
Depending on the sign of $d(\mu)$, we answer two different enumerative questions: 
\begin{enumerate}
    \item If $d(\mu)<0$, $\pi_{1}(\Sigma_{\mu})$ is a proper locus in the moduli space of degree $d$ hypersurfaces, and we compute its equivariant fundamental class in $A^*_{\GL_{r+1}}(\Pp H^0(\mathcal{O}_{\Pp^r}(d)))$.
    \item If $d(\mu)\geqslant 0$, then the locus of $\mu$-incident lines on a general hypersurface of degree $d$ is either empty or of dimension $d(\mu)$, and we compute its fundamental class in $A^*(\mathrm{Gr}(2,r+1))$.
\end{enumerate}
\begin{thm}\label{thm2}
  Let $\mathcal{W}$ be the universal rank $r+1$ vector bundle over $\mathrm{BGL}_{r+1}$ with Chern classes $c_1,\dots,c_{r+1}.$  Let $\gamma=c_1(\mathcal{O}(1))$ on $\Pp(\mathrm{Sym}^d\mathcal{W}^\vee)$, and let $\alpha:\mathrm{Gr}(2,\mathcal{W})\to \mathrm{Pic}_0^d$ 
    be the map corresponding to $\mathcal{O}_{\Pp\mathcal{S}}(d).$ Define $g_{\nu}^i=\alpha^* f_{\nu}^i\in A^{i-n}(\mathrm{Gr}(2,\mathcal{W}))$. Let $m_a$ denote the multiplicity of $a$ in $\nu$.
    The class of the universal $\mu$-incident line is 
    \[[\Sigma_{\mu}]=\frac{1}{\prod_{a}m_a!}\sum_{i=0}^{c(\nu)} \gamma^{c(\nu)-i} g_{\nu}^i.\] 
    \begin{enumerate}
        \item     When $d(\mu)<0,$ the class of $\pi_{1}(\Sigma_{\mu})$ is given by
    \[[V_{\mu}]=\frac{1}{\prod_{a}{m_a!}}\sum_{i=0}^{c(\nu)} \gamma^{c(\nu)-i}\beta_*g_{\nu}^i\in A^*_{\GL_{r+1}}(\Pp H^0(\mathcal{O}_{\Pp^r}(d))),\]
    where $\beta:\mathrm{Gr}(2,\mathcal{W})\to \mathrm{BGL}_{r+1}.$
    \item When $d(\mu)\geqslant 0$, the class of $\mu$-incident lines to a general degree $d$ hypersurface $X$ is given by evaluating $g_{\nu}^{c(\nu)}$ at $c_i=0$:
    \[[S_{\mu}(X)]=\frac{1}{\prod_{a}{m_a!}}g_{\nu}^{c(\nu)}|_{c_i=0}\in A^*(\mathrm{Gr}(2,r+1))\]
    \end{enumerate}
\end{thm}
\begin{ex}
    Continuing Example \ref{intro example}, let $d=6$ and $\mu=\nu=(2,2,2)$. When $r=2$, $d(\mu)=-1$, and $V_{\mu}$ has codimension $1$ in the moduli space of plane sextics. Theorem \ref{thm2} tells us that 
    \[[V_{\mu}]=-624c_1+312\gamma.\]
    However, when $r=3$, $d(\mu)=1$, so there is a one-dimensional family of $\mu$-incident lines. Theorem \ref{thm2} tells us that 
    \begin{align*}
       [S_{\mu}(X)]=624\sigma_{2,1},
    \end{align*}
    where $\sigma_{\lambda}$ denotes the standard Schubert classes on $\mathrm{Gr}(2,4).$ The degree of $[S_{\mu}(X)]$ is the number of tritangent lines on a sextic surface meeting a fixed general line. Equivalently, it is the degree of the scroll of tritangent lines to a general sextic surface, which was computed in \cite[Proposition 4.3]{ciliberto2024varietytritangentialplanesgeneral}.
    
    We note that the equivariant class $[V_{\mu}]$ determines $[S_{\mu}(X)]$ in the following way. Let $X=V(F)$ be a general sextic surface, and let $\ell$ be a general line. The planes containing $\ell$ form a pencil $B\cong \Pp^1.$ Consider the $\Pp^2$-bundle $\Pp \mathcal{U}\to B$ whose fiber over a member of the pencil $[H]\in B$ is the plane $H.$ Since $H$ contains $\ell$, we have $\mathcal{U}\cong \mathcal{O}_{B}^{\oplus 2}\oplus \mathcal{O}_{B}(-1).$ Then, $c_{1}$ pulls back to $c_{1}(\mathcal{U})=c_1(\mathcal{O}_B(-1))=-\zeta.$ Consider the map of vector bundles $\mathcal{O}_{B}\to \mathrm{Sym}^6\mathcal{U}^\vee$ sending $1\mapsto F|_{H}$, which is nowhere zero since $F$ is smooth. This determines a section $s_{F}:B\to \Pp(\mathrm{Sym}^6\mathcal{U}^\vee)$, and the pullback of the hyperplane class along this map is zero. Therefore, 
    \[[V_\mu]|_{B}=624\zeta,\]
    which tells us 
    that there are 624 planes $H$ in the pencil for which the plane sextic $X\cap H$ admits a tritangent line. A tritangent line to a sextic surface meeting a fixed general line $\ell$ determines a unique plane $H$. Conversely, a tritangent line meeting $X\cap H$ must meet $\ell,$ so we conclude that $\deg [V_{\mu}]|_{B}=\deg [S_{\mu}(X)].$
\end{ex}
\begin{ex}
    In the case that the hypersurface $X$ is a plane curve, 
    our formulas for $[V_{\mu}]$ are equivariant enrichments of the formulas in \cite{ciliberto2026contactinvariantsplanecurves}. The following table shows the $\GL_3$-equivariant classes of these loci up to codimension $2$. 
    \begin{xltabular}{\textwidth}{
    >{\centering\arraybackslash}p{1.5cm}
    >{\centering\arraybackslash}p{3 cm}
    >{\raggedright\arraybackslash}X
}
$\nu$ & Codimension & Equivariant Fundamental Class of $V_{\mu}$ \\ 
\midrule

$(4)$ & $1$ &
$\begin{aligned}
    &-2d(3d-2)(d-3)c_1
+
6(3d-2)(d-3)\gamma
\end{aligned}$ 
\\
\midrule

$(3,2)$ & $1$ &
$\begin{aligned}
    &-d(d-4)(d-3)(d^2+6d-4)c_1\\& \qquad
+
3(d-4)(d-3)(d^2+6d-4)\gamma
\end{aligned}$ 
\\
\midrule

$(2,2,2)$ & $1$ &
$\begin{aligned}
    & \frac{1}{6}\Bigl[-2d(d-5)(d-4)(d-3)(d^2+3d-2)c_1
 \\
    &\qquad +
6(d-5)(d-4)(d-3)(d^2+3d-2)\gamma\Bigr]
\end{aligned}$ 
\\
\midrule
$(5)$ & $2$ &
$\begin{aligned}
   & 10d(d-4)(d^2-3d+3)c_1^2
-10d(d-4)(4d-5)\gamma c_1 \\
&\qquad
+15(d-4)(4d-5)\gamma^2
-5d(d-4)(d-2)(2d-9)c_2
\end{aligned}$ 
\\
\midrule
$(3,3)$ & $2$ &
$\begin{aligned}
   &\frac{1}{2}\Bigl[d(d-5)(d-4)(d^3+9d^2-28d+30)c_1^2 \\
&\qquad
-6d(d-5)(d-4)(d^2+7d-9)\gamma c_1 \\
&\qquad
+9(d-5)(d-4)(d^2+7d-9)\gamma^2 \\
&\qquad
-3d(d-5)(d-4)(d-2)(2d-15)c_2\Bigr]
\end{aligned}$ 
\\
\midrule

$(4,2)$ & $2$ &
$\begin{aligned}
&d(d-5)(d-4)(d^3+8d^2-27d+30)c_1^2 \\
&\qquad
-4d(d-5)(d-4)(d^2+11d-14)\gamma c_1 \\
&\qquad
+6(d-5)(d-4)(d^2+11d-14)\gamma^2 \\
&\qquad
-d(d-5)^2(d-4)(d-2)(d+9)c_2
\end{aligned}$ 
\\
\midrule

$(3,2,2)$ & $2$ &
$\begin{aligned}
    &\frac{1}{2}\Bigr[d(d-6)(d-5)(d-4)(d+6)(2d^2-5d+5)c_1^2 \\
&\qquad
-2d(d-6)(d-5)(d-4)(5d^2+23d-30)\gamma c_1 \\
&\qquad
+3(d-6)(d-5)(d-4)(5d^2+23d-30)\gamma^2 \\
&\qquad
-d(d-6)(d-5)(d-4)(d-2)(d^2-45)c_2\Bigl]
\end{aligned}$ 
\\
\bottomrule
\end{xltabular}
\begin{rem}
    If we replace $\mathcal{W}$ with the twist $\mathcal{W}\otimes \mathcal{L}$ for any line bundle $\mathcal{L},$ we have an isomorphism of projective bundles $\Pp \mathcal{W}\cong \Pp (\mathcal{W}\otimes \mathcal{L})$
    sending 
    \begin{align*}
        c_1&\mapsto c_1+3c_1(\mathcal{L})\\
        c_2&\mapsto c_2+2c_1 c_1(\mathcal{L})+3c_1(\mathcal{L})^2,
    \end{align*}
    and 
    \begin{align*}
\gamma=c_1(\mathcal{O}_{\Pp(\mathrm{Sym}^d\mathcal{W}^\vee)}(1))\mapsto \gamma+dc_1(\mathcal{L}).
    \end{align*}
Invariance under this change of variables implies that the coefficient of $c_1$ in $V_{\mu}$ is $-d/3$ times the coefficient of $\gamma$. Hence, in codimension $1$, the non-equivariant fundamental class of $V_{\mu}$ fully determines the equivariant fundamental class. In codimension $2$, any class can be written as 
\[A\gamma^2+Bc_1\gamma+Cc_1^2+Dc_2,\]
and
invariance under the above change of variables implies that 
\[B=-\frac{2d}{3}A \qquad \qquad C=\frac{d^2}{9}A-\frac{1}{3}D.\]
We see that in codimension $2$ and higher, the non-equivariant fundamental class of $V_{\mu}$ does not fully determine the equivariant fundamental class. 
\end{rem}
\end{ex}
In Section \ref{section6}, we also generalize these contact problems to the case of complete intersections in projective space, where we replace $\Pp(\mathrm{Sym}^d\mathcal{W}^\vee)$ with a partial compactification of Di Lorenzo's moduli stack of complete intersections \cite{dilorenzo2022intersectiontheorymodulismooth}. Throughout, we work over a field $k$ of characteristic zero.
\begin{conv}
    For a vector bundle $\mathcal{E}$ on $X$ we define its projectivization \[\pi:\Pp\mathcal{E}\coloneq \underline{\mathrm{Proj}}\left(\mathrm{Sym}^\bullet \mathcal{E}^\vee\right)\to X\] so we have the tautological inclusion $\mathcal{O}_{\Pp\mathcal{E}}(-1)\hookrightarrow\pi^*\mathcal{E}.$
\end{conv}
\subsection*{Acknowledgements}The author would like to thank Hannah Larson for many helpful conversations and feedback regarding this work. This project was inspired by generalizing and equivariantly enriching the $60$ hyperflexes in a pencil of plane quartics, a classical problem the author learned about from Bernd Sturmfels. The author additionally thanks Akiyoshi Sannai for pointing out the need for an additional assumption in Theorem \ref{thm1}. The author is grateful for the support of the NSF Graduate Research Fellowship. The author acknowledges use of ChatGPT 5.6 for refining the proof of Lemma \ref{push forward rule}, for assistance generating SageMath code for computing examples with the weighted bundles of principal parts, and for proofreading. The manuscript was written entirely by the author, who assumes responsibility for all content.
\section{Weighted Bundles of Principal Parts}
Let $f:\Psi\to S$ be a family of genus $g$ curves, by which we mean a smooth proper morphism such that for every $s\in S$, the fiber $\Psi_s$ is a connected genus $g$ curve. Let $\mathcal{L}$ be a relative degree $d$ line bundle on $\Psi$, and fix $\nu=(\nu_1,\dots,\nu_n)$. Our goal is to define a vector bundle $\mathcal{F}_{\nu}(\mathcal{L})$ on $\Psi^n=\Psi\times_S\dots\times_S\Psi$ whose fibers encode the zeros of sections of $\mathcal{L}$. Let $p:\Psi^{n+1}\to \Psi^n$ be the projection onto the first $n$ factors and $p_i:\Psi^{n+1}\to \Psi$ for $1\leqslant i\leqslant n+1$ the projection onto the $i$th factor. We obtain the following Cartesian square: 
\begin{center}
    \begin{tikzcd}
\Psi^{n+1} \arrow[r, "p"] \arrow[d, "p_{n+1}"'] & \Psi^n \arrow[d, "g"] \\
\Psi \arrow[r, "f"']                 & S          
\end{tikzcd}
\end{center}
We define $D_{i,j}$ to be the diagonal where the $i$th and $j$th factors agree. Then, the diagonals $D_{i,n+1}$ determine sections of $p$. Further, we let  $q_j:\Psi^n\to \Psi$ for $1\leqslant j\leqslant n$ be the projection onto the $j$th factor. We define the sheaf 
\[\mathcal{F}_{\nu}(\mathcal{L})=p_{*}\left(p_{n+1}^*\mathcal{L}\otimes \mathcal{O}_{\Psi^{n+1}}/\mathcal{I}_{D_{1,n+1}}^{\nu_1}\cdots\mathcal{I}_{D_{n,n+1}}^{\nu_n}\right),\]
which we call the weighted principal parts bundle with weights $(\nu_1,\dots,\nu_n).$
\begin{prop}\label{properties}
The bundle $\mathcal{F}_{\nu}(\mathcal{L})$ has the following properties.
\begin{enumerate}
        \item $\mathcal{F}_{\nu}(\mathcal{L})$ is a vector bundle of rank $c(\nu)=\sum_{i=1}^n \nu_i$ whose construction commutes with base change on $S$.
        \item There is a natural evaluation map $g^*f_*\mathcal{L}\to \mathcal{F}_{\nu}(\mathcal{L})$.    \end{enumerate}
\end{prop}
\begin{proof}
For $(1)$, over a point $y=(s,p_1,\dots,p_n)\in \Psi^n$, letting $p^{-1}(y)=\Psi_s,$ we have
\[h^0(\Psi_s,\mathcal{L}|_{\Psi_s}\otimes \mathcal{O}_{\Psi_s}/\mathcal{I}_{p_1}^{\nu_1}\dots\mathcal{I}_{p_n}^{\nu_n})=c(\nu),\]
and the result follows from cohomology and base change.
For $(2)$, note that we have the natural map $f^*f_*\mathcal{L}\to \mathcal{L},$ and pullback to get a map $p_{n+1}^*f^*f_*\mathcal{L}\to p_{n+1}^*\mathcal{L}$, and then compose with the quotient map $p_{n+1}^*\mathcal{L}\to p_{n+1}^*\mathcal{L}\otimes \mathcal{O}_{\Psi^{n+1}}/\mathcal{I}_{D_{1,n+1}}^{\nu_1}\dots\mathcal{I}_{D_{n,n+1}}^{\nu_n}$ to get the composition 
\[p_{n+1}^*f^*f_*\mathcal{L}\to p_{n+1}^*\mathcal{L}\otimes \mathcal{O}_{\Psi^{n+1}}/\mathcal{I}_{D_{1,n+1}}^{\nu_1}\dots\mathcal{I}_{D_{n,n+1}}^{\nu_n}.\]
Then, since $f\circ p_{n+1}=g\circ p,$ we have a map 
\[p^*g^*f_*\mathcal{L}\to p_{n+1}^*\mathcal{L}\otimes \mathcal{O}_{\Psi^{n+1}}/\mathcal{I}_{D_{1,n+1}}^{\nu_1}\dots\mathcal{I}_{D_{n,n+1}}^{\nu_n},\]
so by adjunction, we have a natural map 
\[g^*f_*\mathcal{L} \to p_*(p_{n+1}^*\mathcal{L}\otimes \mathcal{O}_{\Psi^{n+1}}/\mathcal{I}_{D_{1,n+1}}^{\nu_1}\dots\mathcal{I}_{D_{n,n+1}}^{\nu_n})\coloneq \mathcal{F}_{\nu}(\mathcal{L}).\]
\end{proof}
\begin{lem}\label{surjective evaluation map}
    The evaluation map in Proposition \ref{properties} is surjective when $\mathcal{L}$ is relatively $(c(\nu)-1)$-very ample.
\end{lem}
\begin{proof}
Let $D=\sum_{i=1}^n\nu_iD_{i,n+1} \subseteq \Psi^{n+1}$, and let $D_{s}=D|_{\Psi_{s}},$ which is length $c(\nu).$ Since $\mathcal{L}|_{\Psi_s}$ is $c(\nu)-1$ very ample, the induced map on cohomology $H^0(\mathcal{L}|_{\Psi_s})\to H^0(\mathcal{L}|_{D_s})$ is surjective.
By cohomology and base change, this is exactly the map on fibers of the evaluation map $g^*f_*\mathcal{L}\to \mathcal{F}_{\nu}(\mathcal{L}).$
\end{proof}
By cohomology of line bundles on $\Pp^1,$ we note that in the case $\Psi\to S$ is a family of genus $0$ curves, $\mathcal{L}$ is relatively $(c(\nu)-1)$-very ample if and only if $d-c(\nu)\geqslant -1$. 
\begin{rem}
    In the case that $\nu=(d),$ $\mathcal{F}_{\nu}(\mathcal{L})$ is a relative bundle of principal parts. We refer the reader to  \cite[Section 11.1]{eisenbud20163264} for the definition and properties of relative bundles of principal parts. Indeed, consider the diagram 
    \begin{center}
        \begin{tikzcd}
\Psi^2 \arrow[r, "p"] \arrow[d, "p_2"'] & \Psi \arrow[d] \\
\Psi \arrow[r]             & S            
\end{tikzcd}
    \end{center}
We have that 
    \[\mathcal{F}_{\nu}(\mathcal{L})=p_*(p_2^*\mathcal{L}\otimes \mathcal{O}_{\Psi^2}/\mathcal{I}_{\Delta}^d)=\mathcal{P}_{\Psi/S}^{d-1}(\mathcal{L}).\]
    
  \end{rem}
\subsection{Chern classes of $\mathcal{F}_{\nu}(\mathcal{L})$}
We now give a formula for the Chern classes of the vector bundle $\mathcal{F}_{\nu}(\mathcal{L})$. 
\begin{thm}\label{chern classes}
    Let $\mathcal{E}_{\nu}(\mathcal{L})=p_*\left({p_{n+1}^*\mathcal{L}\otimes \mathcal{O}_{\Psi^{n+1}}\left(-\sum_{i=1}^n \nu_iD_{i,n+1}\right)}\right).$ Then, $\mathcal{F}_{\nu}(\mathcal{L})$ admits a filtration by subbundles that are kernels of the natural surjections $\mathcal{F}_{\nu}(\mathcal{L})\to \mathcal{F}_{\nu'}(\mathcal{L})$ for $(\nu_1,\dots,\nu_n)\geqslant (\nu_1',\dots,\nu_n').$ The graded pieces of the filtration are identified by the exact sequences 
    \begin{center}
    \begin{tikzcd}
0 \arrow[r] & {\mathcal{E}_{\nu}(\mathcal{L})|_{D_{j,n+1}}} \arrow[r] & {\mathcal{F}_{\nu_1,\dots,\nu_j+1,\dots,\nu_n}}(\mathcal{L}) \arrow[r] & {\mathcal{F}_{\nu}} (\mathcal{L})\arrow[r] & 0
\end{tikzcd}
\end{center}
for $1\leqslant j\leqslant n.$ Then, the total Chern class of $\mathcal{F}_{\nu}(\mathcal{L})$ is given by 
\begin{align*}
    c(\mathcal{F}_{\nu}(\mathcal{L}))&=\prod_{j_1=0}^{\nu_1-1}c(\mathcal{E}_{j_1,0,\dots,0}(\mathcal{L})|_{D_{1,n+1}})\cdots \prod_{j_n=0}^{\nu_n-1}c(\mathcal{E}_{\nu_1,\dots,\nu_{n-1},j_n}(\mathcal{L})|_{D_{n,n+1}})\\&=\prod_{i=1}^n\prod_{j=0}^{\nu_i-1}c(\mathcal{E}_{\nu_1,\dots,\nu_{i-1},j,0,\dots,0}(\mathcal{L})|_{D_{i,n+1}}).
\end{align*}
\end{thm}
\begin{proof}
    Let $\mathcal{E}'_{\nu}(\mathcal{L})=p_{n+1}^*\mathcal{L}\otimes \mathcal{O}_{\Psi^{n+1}}\left(-\sum_{i=1}^n\nu_iD_{i,n+1}\right)$ and 
\[\mathcal{F}_{\nu}'(\mathcal{L})=p_{n+1}^*\mathcal{L}\otimes\mathcal{O}_{\Psi^{n+1}}/\mathcal{I}_{D_{1,n+1}}^{\nu_1}\cdots \mathcal{I}_{D_{n,n+1}}^{\nu_n}.\] Since the diagonals $D_{i,j}$ are effective Cartier divisors, we have the exact sequences
\begin{center}
    \begin{tikzcd}
0 \arrow[r] & \mathcal{E}'_{\nu}(\mathcal{L})\arrow[r] & p_{n+1}^*\mathcal{L} \arrow[r] & {\mathcal{F}'_{\nu}} (\mathcal{L})\arrow[r] & 0
\end{tikzcd}
\end{center}
and 
\begin{center}
    \begin{tikzcd}
0 \arrow[r] & {\mathcal{E}'_{\nu_1,\dots,\nu_{j}+1,\dots,\nu_n}}(\mathcal{L}) \arrow[r] & {\mathcal{E}'_{\nu}}(\mathcal{L}) \arrow[r] & {\mathcal{E}'_{\nu}(\mathcal{L})|_{D_{j,n+1}}} \arrow[r] & 0
\end{tikzcd}
\end{center}
which come from the closed subscheme exact sequence for $D_{j,n+1}.$ We put these sequences together into the following diagram:
\begin{equation}\label{snake diagram}
    \begin{tikzcd}
            & 0 \arrow[r] \arrow[d]                                                           & 0 \arrow[r] \arrow[d]                                         & {\mathcal{E}'_{\nu}(\mathcal{L})|_{D_{j,n+1}}} \arrow[d]           &   \\
0 \arrow[r] & {\mathcal{E}'_{\nu_1,\dots,\nu_{j}+1,\dots,\nu_n}(\mathcal{L})} \arrow[r] \arrow[d] & {\mathcal{E}'_{\nu}(\mathcal{L})} \arrow[r] \arrow[d] & {\mathcal{E}'_{\nu}(\mathcal{L})|_{D_{j,n+1}}} \arrow[d] \arrow[r] & 0 \\
0 \arrow[r] & p_{n+1}^*\mathcal{L} \arrow[d] \arrow[r, equal]                    & p_{n+1}^*\mathcal{L} \arrow[r] \arrow[d]                       & 0 \arrow[d]                                                          &   \\
            & {\mathcal{F}'_{\nu_1,\dots,\nu_{j}+1,\dots,\nu_n}(\mathcal{L})} \arrow[r] \arrow[d]   & {\mathcal{F}'_{\nu}(\mathcal{L})} \arrow[r] \arrow[d] & 0                                                                    &   \\
            & 0                                                                               & 0                                                             &                                                                      &  
\end{tikzcd}
\end{equation}
Then, from the Snake Lemma, we get the exact sequences
 \begin{center}
    \begin{tikzcd}
0 \arrow[r] & \mathcal{E}'_{\nu}(\mathcal{L})|_{D_{j,n+1}} \arrow[r] & \mathcal{F}'_{\nu_1,\dots,\nu_j+1,\dots,\nu_n}(\mathcal{L}) \arrow[r]  & \mathcal{F}'_{\nu} (\mathcal{L})\arrow[r] & 0
\end{tikzcd}
\end{center}
for $1\leqslant j\leqslant n.$
When we pushforward by $p_*:\Psi^{n+1}\to \Psi^n$, we have $R^1p_*(\mathcal{E}'_{\nu}|_{D_{j,n+1}})=0$ since $p|_{D_{j,n+1}}$ is an isomorphism. It then follows that the sequence
 \begin{center}
    \begin{tikzcd}
0 \arrow[r] & \mathcal{E}_{\nu}(\mathcal{L})|_{D_{j,n+1}} \arrow[r] & \mathcal{F}_{\nu_1,\dots,\nu_j+1,\dots,\nu_n}(\mathcal{L}) \arrow[r]  & \mathcal{F}_{\nu}(\mathcal{L}) \arrow[r] & 0
\end{tikzcd}
\end{center}
is exact for $1\leqslant j\leqslant n$. 

Using that $\mathcal{F}_{0,\dots,0}(\mathcal{L})=0$, the filtration above allows us to compute $c(\mathcal{F}_{\nu}(\mathcal{L}))$ by picking a sequence of $n$-tuples $(e_1,\dots,e_n)$ from $(0,\dots,0)\to (\nu_1,\dots,\nu_n)$ that increases the value of one coordinate by one at each step. The formula given in Theorem \ref{chern classes} corresponds to the path that increases $e_1$ to $\nu_1$, and then $e_2$ to $\nu_2$, and so on (i.e.  $(0,\dots,0)\to (1,0,\dots,0)\to \dots \to (\nu_1,0,\dots,0)\to \dots \to (\nu_1,\dots,\nu_{n-1},0)\to \dots \to (\nu_1,\dots,\nu_{n-1},1)\to \dots \to (\nu_1,\dots,\nu_n)$).
\end{proof}
We can more explicitly describe the Chern classes of $\mathcal{F}_{\nu}(\mathcal{L})$ by computing the Chern classes of $\mathcal{E}_{\nu_1,\dots,\nu_n}(\mathcal{L})|_{D_{j,n+1}}.$ Let $i_j:D_{j,n+1}\to \Psi^{n+1}.$ Since $p_{n+1}\circ i_j=q_j$ under the isomorphism $D_{j,n+1}\cong \Psi^n,$ 
we have 
\begin{align*}
i^*_j\left(p_{n+1}^*\mathcal{L}\otimes \mathcal{O}_{\Psi^{n+1}}\left(-\sum_{i=1}^n\nu_iD_{i,n+1}\right)\right)&=q_j^*\mathcal{L}\otimes i_j^*\mathcal{O}_{\Psi^{n+1}}\left(-\sum_{i=1}^n\nu_iD_{i,n+1}\right)\\&=q_j^*\mathcal{L}\otimes \,\bigotimes_{i\neq j}^n\mathcal{O}_{\Psi^n}(-\nu_i D_{ij})\otimes (\mathcal{N}_{i_j}^\vee)^{\otimes \nu_j},
\end{align*}
Hence, 
\begin{equation}\label{chern E}
c(\mathcal{E}_{\nu}(\mathcal{L})|_{D_{j,n+1}})=1+c_1(q_j^*\mathcal{L})-\sum_{\substack{i=1\\i\neq j}}^n\nu_i[D_{ij}]+\nu_jc_1(\mathcal{N}_{i_j}^\vee),
\end{equation}
which leads us to the following formula for $c(\mathcal{F}_{\nu}(\mathcal{L})):$
\begin{prop}\label{Chern formula}
    Let $\zeta_j=c_1(q_j^*\mathcal{L})$, $\delta_{ij}=[D_{ij}],$ and $\psi_j=c_1(\mathcal{N}_{i_j}^\vee)=c_1(q_j^*\omega_{f}).$ Then, substituting equation \eqref{chern E} into the formula in Theorem \ref{chern classes} gives
    \begin{align*}
        c(\mathcal{F}_{\nu}(\mathcal{L}))=&\prod_{j_1=0}^{\nu_1-1}(1+\zeta_1+j_1\psi_1)\prod_{j_2=0}^{\nu_2-1}(1+\zeta_2-\nu_1\delta_{12}+j_2\psi_2)\cdots\\&\prod_{j_3=0}^{\nu_3-1}\left(1+\zeta_3-\sum_{i=1}^2\nu_i\delta_{i3}
+j_3\psi_3\right)\dots\prod_{j_n=0}^{\nu_n-1}\left(1+\zeta_n-\sum_{i=1}^{n-1}\nu_{i}\delta_{in}+j_n\psi_n\right)\\&=\prod_{k=1}^n\prod_{j=0}^{\nu_k-1}\left(1+\zeta_k-\sum_{i<k}\nu_i\delta_{ik}+j\psi_k\right).
    \end{align*}
\end{prop}
\begin{rem}
 The formula above for the Chern classes of $\mathcal{F}_{\nu}(\mathcal{L})$ seems to depend on a  path from $(0,\dots,0)\to (\nu_1,\dots,\nu_n)$; however, all paths yield trivially equivalent formulas. To a path from $(0,\dots,0)\to (\nu_1,\dots, \nu_n)$, associate the step sequence denoting which index was raised. All paths differ by a series of adjacent transpositions in their step sequences, so it suffices to show that $c(\mathcal{F}_{\nu}(\mathcal{L}))$ is invariant under these adjacent transpositions. This amounts to showing 
  \[c(\mathcal{E}_{e_1,\dots,e_n}(\mathcal{L})|_{D_{j,n+1}})c(\mathcal{E}_{e_1,\dots,e_{j}+1,\dots,e_n}(\mathcal{L})|_{D_{i,n+1}})=c(\mathcal{E}_{e_1,\dots,e_n}(\mathcal{L})|_{D_{i,n+1}})c(\mathcal{E}_{e_1,\dots,e_{i}+1,\dots,e_n}(\mathcal{L})|_{D_{j,n+1}}).\]
  From equation \eqref{chern E}, we have 
  \begin{align*}
      & c(\mathcal{E}_{e_1,\dots,e_n}(\mathcal{L})|_{D_{j,n+1}})c(\mathcal{E}_{e_1,\dots,e_{j}+1,\dots,e_n}(\mathcal{L})|_{D_{i,n+1}})-c(\mathcal{E}_{e_1,\dots,e_n}(\mathcal{L})|_{D_{i,n+1}})c(\mathcal{E}_{e_1,\dots,e_{i}+1,\dots,e_n}(\mathcal{L})|_{D_{j,n+1}})\\&=\delta_{ij}\left(\zeta_i-\zeta_j-\sum_{\substack{k=1\\k\neq i}}^n e_k\delta_{ki}+\sum_{\substack{k=1\\k\neq j}}^n e_k\delta_{kj}+e_i\psi_i-e_j\psi_j\right).
  \end{align*}
  Letting $\Delta_{ij}:D_{i,j}\to \Psi^{n}$, and using $\delta_{ij}=\Delta_{ij*}(1)$, the above expression is equal to 
  \begin{align*}
&\Delta_{ij*}\Delta_{ij}^*\left(\zeta_i-\zeta_j-\sum_{\substack{k=1\\k\neq i}}^n e_k\delta_{ki}+\sum_{\substack{k=1\\k\neq j}}^n e_k\delta_{kj}+e_i\psi_i-e_j\psi_j\right)=\\&\Delta_{ij*}\Delta_{ij}^*(\zeta_i-\zeta_j-e_j\delta_{ji}+e_i\delta_{ij}+e_i\psi_i-e_j\psi_j)=0,
  \end{align*}
  since $q_j\circ \Delta_{ij}=q_i\circ \Delta_{ij}$, $\Delta^*_{ij}(\delta_{ij})=c_1(\mathcal{N}_{\Delta_{ij}})$, and $\Delta_{ij}^*\psi_i=\Delta_{ij}^*\psi_j=\Delta_{ij}^*q_{j}^*\omega_{f}=c_{1}(\mathcal{N}_{\Delta_{ij}}^\vee).$
\end{rem}
\section{Universal Polynomials in $A^*(\mathrm{Pic}_g^d)$}
The pair $(f:\Psi\to S,\mathcal{L})$ determines a morphism from $S$ to the universal Picard stack $\mathrm{Pic}_g^d.$ Let $f:\mathcal{C}_g\to \mathrm{Pic}_g^d$ be the pullback of the universal curve over $\mathcal{M}_g$. It is equipped with the universal degree $d$ line bundle $\mathscr{L}$. Since the construction of the weighted bundle of principal parts commutes with base change, there is a universal weighted bundle of principal parts $\mathcal{F}_{\nu}(\mathscr{L})$ on $\mathcal{C}_{g}^n,$ such that whenever we have a Cartesian diagram
\begin{center}
    \begin{tikzcd}
\Psi^n \arrow[d] \arrow[r, "\alpha_n"] & \mathcal{C}_{g}^n \arrow[d] \\
S \arrow[r, "\alpha"']                & \mathrm{Pic}_g^d           
\end{tikzcd}
\end{center}
we have $\mathcal{F}_{\nu}(\mathcal{L})=\alpha_{n}^*(\mathcal{F}_{\nu}(\mathscr{L})).$
\subsection{Combinatorial pushforward lemma}
Let $g:\mathcal{C}_g^{n}\to \mathrm{Pic}_g^d.$
In Theorem \ref{thm1} and Theorem \ref{thm2}, our enumerative formulas are expressed in terms of 
\begin{align}\label{gmu}
g_{\nu}^i=\alpha^*f_{\nu}^i=\alpha^*g_*c_i(\mathcal{F}_{\nu}(\mathscr{L}))
\end{align}
We now describe how to compute $f_{\nu}^i$ as an explicit polynomial in the twisted kappa classes. 
By Proposition \ref{Chern formula}, $c(\mathcal{F}_{\nu}(\mathscr{L}))$ is the sum of monomials in $\zeta_i,\psi_i,\delta_{ij},$ so it suffices to give a formula for pushing forward such a monomial.
\begin{lem}\label{push forward rule}
For a monomial in $\zeta_i,\psi_i,\delta_{ij},$
    draw the multigraph $\Gamma$ on $n$ vertices, where there are $e_{ij}$ edges between $i,j$ if the monomial contains $\delta_{ij}^{e_{ij}}$. For each connected component $C\in \pi_0(\Gamma)$, let $h_1(C)$ denote the loop number of $C$. Then,  
\begin{align*}
    g_*\left(\prod_{i=1}^n \psi_i^{a_i}\zeta_i^{b_i} \prod_{1\leqslant i< j\leqslant n} \delta_{ij}^{e_{ij}}\right)&=g_*\left(\prod_{C\in \pi_0(\Gamma)}\left(\prod_{i\in C} \psi_i^{a_i}\zeta_i^{b_i} \prod_{i<j\in C} \delta_{ij}^{e_{ij}}\right)\right)\\
    &=\prod_{C\in \pi_0(\Gamma)}(-1)^{h_1(C)}\kappa_{\sum\limits_{i \in C} {a_i}+h_1(C)-1,\sum\limits_{i\in C} b_i}.
\end{align*}
\end{lem}
\begin{proof}
For each connected component $C\in \pi_0(\Gamma),$ we denote the diagonal in $\mathcal{C}_g^{n}$ where all points corresponding to vertices in the component $C$ agree by $D_C$, and write $\delta_{C}$ for its class. Additionally, for each connected component $C$, choose a spanning tree $T_C$. We have that $\prod_{e\in T_C}\delta_e=\delta_C.$ The number of edges of $\Gamma$ not included in $T_C$ is
   \[\abs{E(C)\setminus E(T_C)}=\abs{E(C)}-\abs{C}+1=h_1(C).\]
Let $D_{\Gamma}\cong \mathcal{C}_g^{\abs{\pi_0(\Gamma)}}$ be the diagonal where each factor corresponds to a connected component of $\Gamma$, and let $\Delta_{\Gamma}:\mathcal{C}_g^{\abs{\pi_0(\Gamma)}}\to \mathcal{C}_g^n$ be the inclusion of this diagonal.
Let $p_C:\mathcal{C}_g^{\abs{\pi_0(\Gamma)}}\to \mathcal{C}_g$ be the projection onto the factor corresponding to the connected component $C$, and let $q=g\circ \Delta_{\Gamma}.$ Using that 
\[\Delta_{\Gamma}^*(\delta_{e})=-p_C^*\psi,\]
it follows that
\begin{align*}
 g_*\left(\prod_{i=1}^n \psi_i^{a_i}\zeta_i^{b_i}\prod_{1\leqslant i< j\leqslant n}\delta_{ij}^{e_{ij}}\right)&=g_*\left(\prod_{i=1}^n \psi_i^{a_i}\zeta_i^{b_i}\prod_{C\in \pi_0(\Gamma)}\delta_C\prod_{e\in E(C)\setminus E(T_C)}\delta_e\right)\\&=
 g_*\left(\prod_{i=1}^n \psi_i^{a_i}\zeta_i^{b_i}\Delta_{\Gamma*}(1)\left(\prod_{C\in \pi_0(\Gamma)}\prod_{e\in E(C)\setminus E(T_C)}\delta_e\right)\right)\\&=g_*\Delta_{\Gamma_*}\Delta_{\Gamma}^*\left(\prod_{i=1}^n \psi_i^{a_i}\zeta_i^{b_i}\left(\prod_{C\in \pi_0(\Gamma)}\prod_{e\in E(C)\setminus E(T_C)}\delta_e\right)\right)\\&=\prod_{C\in \pi_0(\Gamma)}(-1)^{h_1(C)}g_*\Delta_{\Gamma_{*}}\left(\prod_{C\in \pi_0(\Gamma)}p_C^*\left(\psi^{\sum_{i\in C}a_i+h_1(C)}\zeta^{\sum_{i\in C}b_i}\right)\right)\\&=\prod_{C\in \pi_0(\Gamma)}(-1)^{h_1(C)}q_*\left(\prod_{C\in \pi_0(\Gamma)}p_C^*\left(\psi^{\sum_{i\in C}a_i+h_1(C)}\zeta^{\sum_{i\in C}b_i}\right)\right)\\&=\prod_{C\in \pi_0(\Gamma)}(-1)^{h_1(C)}f_*\left(\psi^{\sum_{i\in C}a_i+h_1(C)}\zeta^{\sum_{i\in C}b_i}\right)\\&=\prod_{C\in \pi_0(\Gamma)}(-1)^{h_1(C)}\kappa_{\sum\limits_{i \in C} {a_i}+h_1(C)-1,\sum\limits_{i\in C} b_i}.
\end{align*} 
\end{proof}
We will use the following proposition in Section \ref{section6}:
\begin{prop}
    \label{collorary}
    In the case that we work on the fiber product \[g:\mathcal{C}_{g}^n\to \mathrm{Pic}_{g}^{d_1}\times_{\mathcal{M}_g}\dots \times_{\mathcal{M}_g}\mathrm{Pic}_{g}^{d_\ell},\] we can compute 
    \[f_{\nu}^{i_1,\dots,i_{\ell}}=g_*(c_{i_1}(\mathcal{F}_{\nu}(\mathscr{L}_{d_1}))\dots c_{i_{\ell}}(\mathcal{F}_{\nu}(\mathscr{L}_{d_{\ell}}))).\] 
Let $\pi:\mathcal{C}_{g}\to \prod_{i=1}^{\ell}\mathrm{Pic}_{g}^{d_i}$. On $\prod_{i=1}^{\ell}\mathrm{Pic}_{g}^{d_i}$, define the mixed twisted kappa classes $\kappa_{i,j_{1},\dots,j_{\ell}}$ which generalize the standard twisted kappa classes to the fiber product: 
\[\kappa_{i,j_{1},\dots,j_{\ell}}=\pi_*\left(c_1(\omega_{\pi})^{i+1}c_{1}\left(\mathscr{L}_{1}\right)^{j_1}\dots c_{1}\left(\mathscr{L}_{\ell}\right)^{j_\ell}\right).\]
Let $q_{i,k}^*c_1(\mathscr{L}_{k})=\zeta_{i,k}$ denote the pullback under the projection \[q_{i,k}:\mathcal{C}^n\times_{\mathcal{M}_g }\prod_{i=1}^{\ell}\mathrm{Pic}_g^{d_i}\to \mathcal{C}\times_{\mathcal{M}_g}\mathrm{Pic}_g^{d_k}.\] Similarly as in Lemma \ref{push forward rule}, it follows that
\begin{align*}
g_*\left(\prod_{i=1}^n\psi_{i}^{a_i}\prod_{i=1}^n\prod_{k=1}^{\ell}\zeta_{i,k}^{b_{i,k}}\prod_{1\leqslant i<j\leqslant n}\delta_{ij}^{e_{ij}}\right)=\prod_{C\in \pi_0(\Gamma)}(-1)^{h_1(C)}\kappa_{\sum\limits_{i\in C} {a_i}+h_1(C)-1,\sum\limits_{i\in C} b_{i,1},\dots,\sum\limits_{i\in C} b_{i,\ell}}.
\end{align*}
\end{prop}
We now prove Theorem \ref{thm1}. Recall that $\mu$ is a partition of $d$ with $\mu=\left(\nu_1,\dots,\nu_n,1^{d-c(\nu)}\right)$ and $\nu_i>1$.
Consider the fiber product 
\begin{center}
    \begin{tikzcd}[row sep=30]
 \Psi^n\times_{S}\Pp(f_*\mathcal{L}) \arrow[d, "\mathrm{pr}_1"'] \arrow[r, "\mathrm{pr}_2"] & \Pp(f_*\mathcal{L}) \arrow[d, "\pi"] \\
\Psi^n \arrow[r, "g"']                                         & S                            
\end{tikzcd}
\end{center}
    and the incidence space 
    \begin{align}\label{incidence substack1}
        \widetilde{X}_{\mu}=\left\{(C,s,p_1,\dots,p_n):\mathrm{div}(s)\geqslant \sum_{i=1}^n \nu_ip_i\right\}\subseteq \Psi^n\times_S \Pp(f_*\mathcal{L}).
    \end{align}
\begin{lem}\label{reduced fiber}
   Let $\mathrm{pr}_2|_{\widetilde{X}_{\mu}}:\widetilde{X}_{\mu}\to \Pp(f_*\mathcal{L})$. If $x\in X_{\mu}^{\circ},$ then the fiber ${\mathrm{pr}_{2}}|_{\widetilde{X}_{\mu}}^{-1}(x)$ is reduced.
\end{lem}
\begin{proof}
    Suppose $x=(C,s)\in \Pp(f_*\mathcal{L}).$ The fiber ${\mathrm{pr}_{2}}|_{\widetilde{X}_{\mu}}^{-1}(x)$ is reduced if and only if \[\dim\left(T_{(p_1,\dots,p_n)}\left({\mathrm{pr}_{2}}|_{\widetilde{X}_{\mu}}^{-1}(x)\right)\right)=0\] for all $(p_1,\dots,p_n)\in {\mathrm{pr}_{2}}|_{\widetilde{X}_{\mu}}^{-1}(x).$ We first consider 
    \[T_{(p_1,\dots,p_n)}(C^n)=\Hom_{(p_1,\dots,p_n)} \left(\Spec \frac{k[\varepsilon]}{\langle \varepsilon^2\rangle},C^n\right)=\bigoplus_{i=1}^n T_{p_i}(C).\]
   Let $t_i$ generate the maximal ideal of $\mathcal{O}_{C,p_i}.$ Since $\mathcal{L}$ is locally free, $\mathcal{L}_{p_i}=\mathcal{O}_{C,p_i},$ and $s$ has a power series expansion $f(t_i)=\sum_{j\geqslant 0}c^{(i)}_jt_i^j$ in $\widehat{\mathcal{O}}_{C,p_i}=k[\![t_i]\!].$ An element of $\bigoplus_{i=1}^n T_{p_i}(C)$
    corresponding to maps 
    \begin{align*}
       \varphi_i: \mathcal{O}_{C,p_i}&\to \frac{k[\varepsilon]}{\langle \varepsilon^2\rangle}\\
        t_i&\mapsto a_i\varepsilon
    \end{align*}
    belongs to $T_{(p_1,\dots,p_n)}\left({\mathrm{pr}_{2}}|_{\widetilde{X}_{\mu}}^{-1}(x)\right)$ if and only if the image of the $j$th derivative of $f(t_i)$ vanishes under $\varphi_i$ for $0\leqslant j\leqslant \nu_{i}-1$ and $1\leqslant i\leqslant n$. In particular, when $j=\nu_i-1,$ vanishing under $\varphi_i$ means that 
    \[\varphi_i\left(\frac{d^{\nu_i-1}}{dt^{\nu_i-1}}f(t_i)\right)=\varphi_i\left((\nu_i-1)!c^{(i)}_{\nu_i-1}+\nu_i!c^{(i)}_{\nu_i}t_i\right)=(\nu_i-1)!c^{(i)}_{\nu_i-1}+(\nu_i!)a_ic^{(i)}_{\nu_i}\varepsilon=0.\]
    Since $x\in X_{\mu}^{\circ},$ we have that $c^{(i)}_{\nu_i-1}=0$ and $c^{(i)}_{\nu_i}\neq 0$. Using that $\mathrm{char}\,k=0,$ we conclude that $a_i=0$ for $1\leqslant i\leqslant n$, which completes the proof that $T_{(p_1,\dots,p_n)}\left({\mathrm{pr}_{2}}|_{\widetilde{X}_{\mu}}^{-1}(x)\right)=0.$
\end{proof}
\subsection{Proof of Theorem \ref{thm1}} 
\begin{proof}
    We compute the class of $\widetilde{X}_{\mu}$ by expressing it as the vanishing locus of a map of vector bundles.
    Let $\pi:\Pp(f_*\mathcal{L})\to S$. Consider the pullback of the composition of the tautological inclusion $\mathcal{O}_{\Pp(f_*\mathcal{L})}(-1)\to \pi^*f_*\mathcal{L}$, and the evaluation map  \ref{properties} $g^*f_*\mathcal{L}\to \mathcal{F}_{\nu}(\mathcal{L}):$
    \begin{align}\label{map of vb}
        \mathrm{pr}_2^*\mathcal{O}_{\Pp(f_*\mathcal{L})}(-1)\to \mathrm{pr}_2^*\pi^*f_*\mathcal{L}\cong \mathrm{pr}_1^*g^*f_*\mathcal{L}\to \mathrm{pr}_1^*\mathcal{F}_{\nu}(\mathcal{L}).
    \end{align}
    On fibers, the above map sends $\langle s\rangle \subseteq H^0(C,\mathcal{L}|_{C})$ to its image in \[H^0(C,\mathcal{L}|_{C}\otimes \mathcal{O}_{C}/\mathcal{I}_{p_1}^{\nu_1}\dots \mathcal{I}_{p_n}^{\nu_n}),\]
    so $\widetilde{X}_{\mu}$ is the vanishing locus of the above map of vector bundles. Therefore, when $\widetilde{X}_{\mu}$ has codimension $c(\nu)$, its class is given by 
    \[[\widetilde{X}_{\mu}]=c_{\mathrm{top}}(\mathcal{O}_{\Pp(f_*\mathcal{L})}(1)\boxtimes \mathcal{F}_{\nu}(\mathcal{L}))=\sum_{i=0}^{c(\nu)}\mathrm{pr}_2^*\gamma^{c(\nu)-i}\mathrm{pr}_1^*c_i(\mathcal{F}_{\nu}(\mathcal{L})).\]
Let $m_a$ be the multiplicity of $a$ in the partition $\nu$. Since we assume $\dim X_{\mu'}<\dim X_{\mu}$ for all coarsenings $\mu'$ of $\mu$,  the generic point of every irreducible component of $X_{\mu}$ lies in $X_{\mu}^{\circ}.$ Then, it follows by Lemma  \ref{reduced fiber} that $\mathrm{pr}_2|_{\widetilde{X}_{\mu}}$ is generically finite of degree $\prod_{a} m_a!,$ and 
   \begin{align*}
       \left(\prod_{a} m_a!\right)[X_{\mu}]=\mathrm{pr}_{2*}[\widetilde{X}_{\mu}]&=\mathrm{pr}_{2*}\left(\sum_{i=0}^{c(\nu)}\mathrm{pr}_2^*\gamma^{c(\nu)-i}\mathrm{pr}_1^*c_i(\mathcal{F}_{\nu}(\mathcal{L}))\right)\\&
       =\sum_{i=0}^{c(\nu)}\gamma^{c(\nu)-i}\mathrm{pr}_{2*}\mathrm{pr}_1^*c_i(\mathcal{F}_\nu(\mathcal{L}))\\&=\sum_{i=n}^{c(\nu)}\gamma^{c(\nu)-i}\pi^*g_*c_i(\mathcal{F}_{\nu}(\mathcal{L})) \qquad \text{(since $g$ has relative dimension $n$)}\\&=
       \sum_{i=n}^{c(\nu)}\gamma^{c(\nu)-i}\pi^*\alpha^*f_{\nu}^i \\&=\sum_{i=n}^{c(\nu)}\gamma^{c(\nu)-i}\pi^*g_{\nu}^i.
   \end{align*}
\end{proof}
\section{Specializations of the universal formulas}
\subsection{Strata of differentials}
Let $f:\mathcal{C}\to \mathcal{M}_g$ be the universal curve, and $\mathcal{L}=\omega_{f}^{\otimes k}$, which has relative degree $k(2g-2)$. The $k$-Hodge bundle is the pushforward $\mathcal{H}^k=f_*\omega_{f}^{\otimes k}.$
The stratum of unordered $k$-differentials corresponding to $\mu$ is 
\[\mathcal{P}^k(\mu)=\left\{(C,\eta)\in\Pp(\mathcal{H}^k):\text{ there exist distinct points } (p_1,\dots,p_m) \text{ such that }\mathrm{div}\,\eta=\sum_{i=1}^m\mu_ip_i\right\}.\]
The class of its closure, $\overline{\mathcal{P}^k(\mu)},$ is computed by Theorem \ref{thm1}. To compute $g_{\nu}^i,$
consider the diagram 
\begin{center}
    \begin{tikzcd}
\mathcal{C} \arrow[d, "f"'] \arrow[r, "\beta"] & \mathcal{C}_g \arrow[d, "\pi"]    \\
\mathcal{M}_g \arrow[r, "\alpha"']           & \mathrm{Pic}_g^{k(2g-2)}
\end{tikzcd}
\end{center}
The calculation below shows the twisted kappa classes pull back to multiples of the standard kappa classes: 
\begin{align*}
    \alpha^*(\kappa_{i,j})&=\alpha^*\pi_*(c_1(\omega_{\pi})^{i+1}c_1(\mathscr{L})^j)\\&=f_*\beta^*(c_1(\omega_{\pi})^{i+1}c_1(\mathscr{L})^j)\\&=f_*(c_1(\omega_{f})^{i+1}c_1(\omega_{f}^{\otimes k})^j)\\&=k^j f_*(c_1(\omega_{f}^{i+j+1}))\\&=k^j\kappa_{i+j}.
\end{align*}
\begin{rem}\label{coincident root loci and strata of differentials}
    The case $g=0$ is interesting to consider using this approach. In this case, if we consider the negative $k$-Hodge bundle, $\mathcal{H}^{-k}=f_*\omega_{f}^{\otimes -k},$ by cohomology and base change, the fibers are isomorphic to $H^0(\Pp^1,\mathcal{O}_{\Pp^1}(2k)).$ When $g=0$, $\mathcal{M}_g=\mathrm{BPGL}_2$ and the map $\GL_2\to \mathrm{PGL}_2$ induces a map of classifying stacks $\mathrm{BGL}_2\to \mathrm{BPGL}_2.$ Pulling back the universal curve under this map gives the projectivization of the universal rank $2$ vector bundle $\pi:\Pp\mathcal{W}\to \mathrm{BGL}_2.$ Taking the determinant of the Euler exact sequence 
    \begin{center}
        \begin{tikzcd}
0 \arrow[r] & \Omega_{\Pp\mathcal{W}/\mathrm{BGL}_2} \arrow[r] & \pi^*\mathcal{W}^\vee\otimes \mathcal{O}_{\Pp\mathcal{W}}(-1) \arrow[r] & \mathcal{O}_{\Pp\mathcal{W}} \arrow[r] & 0
\end{tikzcd}
    \end{center}
    yields 
    \[\omega_{\pi}=\mathcal{O}_{\Pp\mathcal{W}}(-2)\otimes \pi^*\det(\mathcal{W}^\vee),\]
    so 
    \[\mathcal{H}^{-k}|_{\mathrm{BGL}_2}=\pi_*((\mathcal{O}_{\Pp\mathcal{W}}(-2)\otimes \pi^*\det(\mathcal{W}^\vee))^{\otimes -k})=\pi_*\mathcal{O}_{\Pp\mathcal{W}}(2k)\otimes \det\mathcal{W}^{\otimes k}=\mathrm{Sym}^{2k}\mathcal{W}^\vee\otimes \det\mathcal{W}^{\otimes k}.\]
    It follows that we have an isomorphism of projective bundles
    \[\Pp(\mathcal{H}^{-k}|_{\mathrm{BGL}_2})\cong \Pp(\mathrm{Sym}^{2k}\mathcal{W}^\vee\otimes \det\mathcal{W}^{\otimes k})\cong \Pp(\mathrm{Sym}^{2k}\mathcal{W}^\vee).\]
    In this way, we see that the classes of closures of strata of $-k$ differentials in genus $0$ are equivalent to classes of $\GL_2$-equivariant coincident root loci of degree $2k$ binary forms. 
    \end{rem}    
\subsection{Coincident root loci of binary forms}
We can use Theorem \ref{thm1} to compute the $\GL_2$-equivariant classes of coincident root loci of degree $d$ binary forms in the same way as the classes of strata of differentials, but by instead pulling back along the map $\alpha:\mathrm{BGL}_2\to \mathrm{Pic}_0^d$ corresponding to $(f:\Pp\mathcal{W}\to \mathrm{BGL}_2,\mathcal{O}_{\Pp\mathcal{W}}(d)).$ 
Letting $\zeta=c_1(\mathcal{O}_{\Pp\mathcal{W}}(1)),$ we have
\begin{align}\label{pullback for coincident root loci}
    \alpha^*(\kappa_{i,j})=f_*(c_1(\omega_{f})^{i+1}c_1(\mathcal{O}_{\Pp\mathcal{W}}(d))^j)=f_*((-2\zeta-c_1)^{i+1}(d\zeta)^j).
\end{align}
In genus $0$, $\mathrm{Pic}_{0}^d$ depends on the parity of $d$. If $d$ is even, $\mathrm{Pic}_0^d$ is equivalent to $\mathrm{BPGL}_2\times \mathrm{B}\mathbb{G}_m$ (see \cite[Remark 2.1]{larson2024chowringuniversalpicard}), and when $d$ is odd, it turns out $\mathrm{Pic}_0^d$ is equivalent to $\mathrm{BGL}_2$. Indeed, when $d$ is odd, a degree $d$ line bundle $\mathcal{L}$ on a family of genus $0$ curves $f:\Psi\to S$ corresponds to the rank $2$ vector bundle $f_*(\mathcal{L}\otimes \omega_{f}^{\otimes (d-1)/2})$ on $S$. Conversely, the rank $2$ vector bundle $\mathcal{V}$ on $S$ corresponds to the family of genus $0$ curves $\Pp\mathcal{V}\to S$ and the line bundle $\mathcal{L}=\mathcal{O}_{\Pp\mathcal{V}}(1)\otimes \omega_{f}^{\otimes -(d-1)/2}.$
\begin{rem}
    The fact that $\alpha$ is an isomorphism when $d$ is odd implies that the rational Chow ring of $\mathrm{Pic}_0^d$ is generated by the $\kappa_{i,j}$. Under this isomorphism, \[\mathscr{L}\mapsto \mathcal{O}_{\Pp \mathcal{W}}(1)\otimes \omega_{\Pp\mathcal{W}/\mathrm{BGL}_2}^{-(d-1)/2}=\mathcal{O}_{\Pp \mathcal{W}}(d)\otimes \pi^*\det(\mathcal{W})^{(d-1)/2},\] and we compute that $\kappa_{0,1}=c_1$ and $\kappa_{2,0}=-2c_1^2+8c_2$. Since $A^*(\mathrm{BGL}_2)_{\Q}=\Q[c_1,c_2],$ this shows that $\kappa_{0,1}$ and $\kappa_{2,0}$ generate $A^*(\mathrm{Pic}_0^d)_{\Q}$ when $d$ is odd. Additionally, completing the square on the projective bundle relation allows us to compute all of the relations between the $\kappa_{i,j},$ and hence the entire presentation of $A^*(\mathrm{Pic}_0^d)_{\Q}$. For $g>2,$ neither the tautological subring nor the full Chow ring of $\mathrm{Pic}_g^d$ is known.
\end{rem}
\begin{rem}
    The $\GL_2$-equivariant class $[X_{\mu}]$ is equal to the projective Thom polynomial, or equivariant Poincar\'e dual of the coincident root loci, which was computed by Feh\'er, N\'emethi, and Rim\'anyi  in \cite{feher2003coincidentrootlocibinary} through a different approach. Additionally, they compute the Thom polynomial $\mathrm{Tp}_{\mu}$ of the coincident root loci, which is given by the $\gamma^0$ term of $[X_{\mu}]$ viewed as an element of $A^*(\mathrm{BGL}_2).$ Although it seems that the projective Thom polynomial of $X_{\mu}$ contains more information than the affine Thom polynomial, it turns out that one can be derived from the other, as noted in \cite{feher2005degeneracy}.
\end{rem}
\section{Equivariant Contact Problems}
\subsection{The universal $\mu$-incident line} We now apply the weighted bundles of principal parts to equivariant contact problems. Recall that $\mu$ is a partition of $d$ with $\mu=\left(\nu_1,\dots,\nu_n,1^{d-c(\nu)}\right)$ and $\nu_i>1$. Our first step in proving Theorem \ref{thm2} is to construct the universal $\mu$-incident line $\Sigma_{\mu}$ and compute its class. 
For this purpose, we introduce the marked universal $\mu$-incident line
$\widetilde{\Sigma}_{\mu}.$ Let $\mathcal{W}$ be the universal rank $r+1$ vector bundle over $\mathrm{BGL}_{r+1},$ and $\beta:\mathrm{Gr}(2,\mathcal{W})\to \mathrm{BGL}_{r+1}$. We write $\mathcal{S}\subseteq \beta^*\mathcal{W}$ for the tautological subbundle of $\mathrm{Gr}(2,\mathcal{W})$. Then, $\widetilde{\Sigma}_{\mu}$ 
is the closed substack of $\Pp(\mathrm{Sym}^d\mathcal{W}^\vee)\times_{{\BGL}_{r+1}}(\Pp\mathcal{S})^n$ whose points over a scheme $T$ are 
      \[
(T \to \Spec k) \longmapsto
\left\{
\begin{array}{l}
(\mathcal{V},\mathcal{U},\sigma_1,\dots,\sigma_n,X) \text{ where } \mathcal V
\text{ is a vector bundle of rank } r+1
\text{ on } T, \\[2pt]
\mathcal{U}\subseteq \mathcal{V} 
\text{ is a rank $2$ subbundle}, \sigma_i:T\to \Pp\mathcal{U} \text{ are sections,} \\[2pt]
X\subseteq \Pp\mathcal{V} \text{ is a degree $d$ hypersurface, and } D_{\nu}\subseteq  X\cap \Pp\mathcal{U}
\end{array}
\right\}/\sim ,\]
where the intersection denotes scheme-theoretic intersection, and $D_{\nu}=\sum_{i=1}^n\nu_iD_i,$ where $D_i$ is the divisor defined by the image of the section $\sigma_i.$ 
\begin{rem}
    $\widetilde{\Sigma}_{\mu}$ is a generalization of the substack $\widetilde{X}_{\mu}$ for coincident root loci (i.e. the relative de Jonqui\`eres loci when $f:\Pp\mathcal{W}\to \mathrm{BGL}_2, \mathcal{L}=\mathcal{O}_{\Pp\mathcal{W}}(d)$).
 The interpretation is the following: in the case $r=1$, $\mathrm{Gr}(2,r+1)=\mathrm{Gr}(2,2)=\Spec k,$ so $\mathrm{Gr}(2,\mathcal{W})=\mathrm{BGL}_2$, and since  $\mathcal{V}$ has rank $2$, there is no choice of $\mathcal{U}$ in this case. Hence, the $\GL_2$-equivariant coincident root loci are the $r=1$ case of the loci we compute for degree $d$ hypersurfaces.
\end{rem}
We now prove the formula for the class of $\Sigma_{\mu}$ in Theorem \ref{thm2} by expressing $\widetilde{\Sigma}_{\mu}$ as the zero locus of a section of a vector bundle. We let $f:\Pp\mathcal{S}\to \mathrm{Gr}(2,\mathcal{W})$ and  $\mathcal{L}=\mathcal{O}_{\Pp\mathcal{S}}(d).$
\begin{prop}\label{vanishing locus}
  Let $\gamma=c_1(\mathcal{O}(1))$ on $\Pp(\mathrm{Sym}^d\mathcal{W}^\vee)$.  $\widetilde{\Sigma}_{\mu}$ has codimension $c(\nu)$, and its class is
    \[[\widetilde{\Sigma}_{\mu}]=\sum_{i=0}^{c(\nu)} \gamma^{c(\nu)-i} c_{i}(\mathcal{F}_{\nu}(\mathcal{L}))\in A^*((\Pp\mathcal{S})^n\times_{\mathrm{BGL}_{r+1}}\Pp(\mathrm{Sym}^d\mathcal{W}^\vee)).\] 
\end{prop}
\begin{proof}
  Consider the diagram  \begin{center}
\begin{tikzcd}[row sep=50]
(\Pp\mathcal{S})^n\times_{\mathrm{BGL}_{r+1}}\Pp(\mathrm{Sym}^d\mathcal{W}^\vee) \arrow[d, "\widetilde{\pi}_2"'] \arrow[r, "\widetilde{\pi}_1"]                 & \Pp(\mathrm{Sym}^d\mathcal{W}^\vee) \arrow[d, "\widetilde{\varphi}_2"] \\
{(\Pp\mathcal{S})^n} \arrow[r, "\widetilde{\varphi}_1"'] & \mathrm{BGL}_{r+1}                                     
\end{tikzcd}
\end{center}
and note that $\widetilde{\varphi}_1$ factors as 
\begin{center}
    \begin{tikzcd}
(\Pp\mathcal{S})^n \arrow[r, "g"] & {\mathrm{Gr}(2,\mathcal{W})} \arrow[r, "\beta"] & \mathrm{BGL}_{r+1}.
\end{tikzcd}
\end{center}
We have the tautological inclusion $\mathcal{S}\hookrightarrow \beta^*\mathcal{W}.$ Dualizing and taking $\mathrm{Sym}$ gives us a surjection $\beta^*\mathrm{Sym}^d\mathcal{W}^\vee \to \mathrm{Sym}^d\mathcal{S}^\vee=f_*\mathcal{L}.$ We consider the composition
\[\widetilde{\pi}_1^*\mathcal{O}(-1)\to \widetilde{\pi}_1^*\widetilde{\varphi}_2^*\mathrm{Sym}^d\mathcal{W}^\vee\cong \widetilde{\pi}_2^*\widetilde{\varphi}_1^*\mathrm{Sym}^d\mathcal{W}^\vee\to \widetilde{\pi}_2^*g^*f_*\mathcal{L}\to \widetilde{\pi}_2^*\mathcal{F}_{\nu}(\mathcal{L}),\]
which on fibers sends $f\mapsto f|_{L}\mapsto f|_{D_\nu}.$ Hence, $\widetilde{\Sigma}_{\mu}$ is the vanishing locus of this map, so when it has expected codimension, its class is given by 
\begin{align*}
    [\widetilde{\Sigma}_{\mu}]&=c_{c(\nu)}(\mathcal{O}(1)\boxtimes \mathcal{F}_{\nu}(\mathcal{L}))\\&=\sum_{i=0}^{c(\nu)}\gamma^{c(\nu)-i}c_i(\mathcal{F}_{\nu}(\mathcal{L})).
\end{align*}
We now show $\widetilde{\Sigma}_{\mu}$ always has expected codimension.
By Lemma \ref{surjective evaluation map}, $\widetilde{\pi}_2^*\widetilde{\varphi}_1^*\mathrm{Sym}^d\mathcal{W}^\vee\to \widetilde{\pi}_2^*\mathcal{F}_{\nu}(\mathcal{L})$ is surjective. Let $\widetilde{\pi}_2^*\mathcal{K}$ be the kernel of this surjection. Then, $\widetilde{\Sigma}_{\mu}=\Pp\mathcal{K}$. By \cite[Proposition 9.13]{eisenbud20163264},  
\[\codim \Pp\mathcal{K}=\mathrm{rk}(\widetilde{\varphi}_1^*\mathrm{Sym}^{d}\mathcal{W}^\vee)-\mathrm{rk}(\mathcal{K})=\mathrm{rk}(\mathcal{F}_{\nu}(\mathcal{L}))=c(\nu).\]
\end{proof}
\begin{cor}\label{universal mu incident}
     Let $m_a$ be the multiplicity of $a$ in the partition $\nu$. The class of the universal $\mu$-incident line is 
    \[[\Sigma_{\mu}]={\frac{1}{\prod_{a}m_a!}}\sum_{i=0}^{c(\nu)}\gamma^{c(\nu)-i}g_{\nu}^i.\]
\end{cor}
\begin{proof}
    We have that $(g\times \mathrm{id})|_{\widetilde{\Sigma}_{\mu}}$ is generically finite of degree $\prod_{a}m_a!$, and 
    \[[\Sigma_{\mu}]=\frac{1}{\prod_{a}m_a!}(g\times \mathrm{id})_*[\widetilde{\Sigma}_{\mu}].\] Then, the result follows from the push-pull formula and flat base change.
\end{proof}
The following lemma shows that our formula for $[S_{\mu}]$ in Theorem \ref{thm2} always computes the class of the locus of $\mu$-incident lines.
\begin{lem}\label{fiber dim}
    In the case of nonnegative expected dimension, a general fiber of $\pi_1|_{\Sigma_{\mu}}$ has dimension $d(\mu)$ if and only if the expected formula for $[S_{\mu}]$ is non-zero.
\end{lem}
\begin{proof}
Suppose a general fiber of $\pi_1|_{\Sigma_{\mu}}$ is non-empty, and hence of dimension $d(\mu)$. Then, there exists a $\mu$-incident line, so $S_{\mu}\neq \emptyset.$ Since $\mathrm{Gr}(2,r+1)$ is projective, $[S_{\mu}]\neq 0$. Conversely, suppose that \[[S_{\mu}]=g_{\nu}^{c(\nu)}|_{c_i=0}\neq 0.\]
In particular, this means that $g_{\nu}^{c(\nu)}\neq 0$, which we claim implies that $\pi_1|_{\Sigma_{\mu}}$ is surjective. Indeed, by \cite[Tag 04ZR]{stacks-project}, we may consider $\pi_1|_{\Sigma_\mu}$ before the $\GL_{r+1}$ quotient. Given $F\in H^0(\mathcal{O}_{\Pp^{r}}(d))$, we get an associated section $\sigma_{F}\in H^0(\mathcal{F}_{\nu}(\mathcal{L})).$ Suppose for contradiction that there exists $F$ such that $V(\sigma_F)=\emptyset.$  Then, $\sigma_F$ is nowhere vanishing, and we have the exact sequence 
\begin{center}
    \begin{tikzcd}
0 \arrow[r] & \mathcal{O}_{(\Pp\mathcal{S})^n} \arrow[r, "1\mapsto \sigma_F"] & \mathcal{F}_{\nu}(\mathcal{L}) \arrow[r] & \mathcal{Q} \arrow[r] & 0
\end{tikzcd}
\end{center}
By the Whitney formula, $c(\mathcal{F}_{\nu}(\mathcal{L}))=c(\mathcal{Q}),$ and since $\mathrm{rk}\,\mathcal{Q}=c(\nu)-1$, the top Chern class of $\mathcal{F}_{\nu}(\mathcal{L}),$ and therefore its pushforward by $g$, is zero, yielding a contradiction. Knowing $\pi_1|_{\Sigma_{\mu}}$ is surjective, we may apply the fiber dimension theorem \cite[Tag 05F6]{stacks-project}, which gives us that a general fiber has dimension $d(\mu).$
\end{proof}
\subsection{Proof of Theorem~\ref{thm2}}
\begin{proof}
For the first part, suppose $d(\mu)<0.$ From the formula for $[\Sigma_{\mu}]$ in Corollary \ref{universal mu incident} it follows that
\begin{align*}
   \left(\prod_{a}m_a!\right)[V_{\mu}]&= \left(\prod_{a}m_a!\right)\pi_{1*}[\Sigma_{\mu}]\\&=\pi_{1*}\left(\sum_{i=0}^{c(\nu)}\pi_1^*\gamma^{c(\nu)-i}\pi_2^*g_{\nu}^i\right)\\&=\sum_{i=0}^{c(\nu)}\gamma^{c(\nu)-i}\pi_{1*}\pi_2^*g_{\nu}^i\\&=\sum_{i=0}^{c(\nu)}\gamma^{c(\nu)-i}\varphi_2^*\varphi_{1*}g_{\nu}^i\\&=\sum_{i=0}^{c(\nu)}\gamma^{c(\nu)-i}\beta_*g_{\nu}^i.
\end{align*}
For the second part, suppose $d(\mu)\geqslant 0,$ and let $N=\mathrm{rk}(\mathrm{Sym}^d\mathcal{W}^\vee)-1.$ The locus $S_{\mu}(X)$ fits into the following diagram 
\begin{center}
\begin{tikzcd}[row sep= 45]
S_{\mu}(X) \arrow[d, hook]                                                & \Sigma_{\mu} \arrow[d, hook]                                                                                  &                                        \\
{\mathrm{Gr}(2,r+1)} \arrow[d] \arrow[r] \arrow[rr, dashed, bend left, "\tau"] & {\Pp(\mathrm{Sym}^d\mathcal{W}^\vee)\times_{\mathrm{BGL_{r+1}}}\mathrm{Gr}(2,\mathcal{W})} \arrow[d] \arrow[r] & {\mathrm{Gr}(2,\mathcal{W})} \arrow[d] \\
\Spec k \arrow[r]                                                      & \Pp(\mathrm{Sym}^d\mathcal{W}^\vee) \arrow[r]                                                                 & \mathrm{BGL}_{r+1}                    
\end{tikzcd}
\end{center}
where $\Spec k$ is the point corresponding to a  general degree $d$ hypersurface $X=V(F)$. Let $\tau:\mathrm{Gr}(2,r+1)\to \mathrm{Gr}(2,\mathcal{W})$. We have that $\pi_1^{-1}(F)$ is empty or has expected codimension. Hence, by  Lemma \ref{fiber dim}, its class is given by
\begin{align*}
    \left(\prod_{a}m_a!\right)[\pi_{1}^{-1}(F)]&= \left(\prod_{a}m_a!\right)\pi_{2*}([\Sigma_{\mu}]\cdot \pi_1^*\gamma^{N})\\&=\pi_{2*}\left(\sum_{i=0}^{c(\nu)}\pi_1^*\gamma^{c(\nu)-i+N}\pi_2^*g_{\nu}^i\right)\\&=\sum_{i=0}^{c(\nu)}\pi_{2*}\pi_1^*\gamma^{c(\nu)-i+N}g_{\nu}^i\\&=\sum_{i=0}^{c(\nu)}\varphi_1^*\varphi_{2*}\gamma^{c(\nu)-i+N}g_{\nu}^i.
\end{align*}
Then, the class of $\mu$-incident lines is
\begin{align*}
    \left(\prod_{a}m_a!\right)[S_{\mu}(X)]&=\left(\prod_{a}m_a!\right)\tau^*[\pi_{1}^{-1}(F)]\\&= \tau^*\left(\sum_{i=0}^{c(\nu)}\varphi_1^*\varphi_{2*}\gamma^{c(\nu)-i+N}g_{\nu}^i\right)\\&=g_{\nu}^{c(\nu)}|_{c_i=0}.
\end{align*}
\end{proof}
We note that in the above formulas, $g_{i}^\nu=\alpha^*f_{i}^{\nu}$ is determined by the pullback formula
\begin{align}\label{pullback grassmannian}
    \alpha^*(\kappa_{i,j})=f_*(c_1(\omega_{f})^{i+1}c_1(\mathcal{O}_{\Pp\mathcal{S}}(d))^j)=f_*((-2\zeta+\sigma_1)^{i+1}(d\zeta)^j),
\end{align}
where $\zeta=c_1(\mathcal{O}_{\Pp\mathcal{S}}(1)).$
\begin{ex}
Let $d=4$, $\mu=\nu=(2,2)$ and $r=2$, so $d(\mu)=0.$
From Theorem \ref{thm2}, we have 
\[[S_{\mu}(X)]=\frac{1}{2}g_{\nu}^4|_{c_i=0}.\]
Applying the formula in Proposition \ref{Chern formula} for the Chern classes of $\mathcal{F}_{\nu}(\mathscr{L})$ followed by the pushforward formula in Lemma
\ref{push forward rule} we get 
\[f_{\nu}^4=(\kappa_{-1,2}+\kappa_{0,1})^2-4\kappa_{-1,3}-10\kappa_{0,2}-6\kappa_{1,1}.\]
By \eqref{pullback grassmannian}, we have 
\[g_{\nu}^4=24\sigma_{1}^2+32\sigma_{1,1}=56\sigma_{1,1},\]
so
\[[S_{\mu}(X)]=28\sigma_{1,1}.\]
Indeed, the degree of $S_{\mu}(X)$ computes the classical $28$ bitangents on a smooth plane quartic.
\end{ex}
\subsection{Generalization to complete intersections}\label{section6}
We can extend Theorem \ref{thm2} to complete intersections in projective space. In the case of a degree $d$ hypersurface, any line meets the hypersurface at $d$ points, counted with multiplicity. However, the same does not apply to complete intersections of codimension higher than $1$. Hence, in this subsection, we work with $\mathcal{F}_{\mu}(\mathcal{L})$, where $\mu=(\mu_1,\dots,\mu_m)$ is any tuple of positive integers satisfying $c(\mu)\leqslant d_j+1$ for each degree $d_j$ of the complete intersection.

Let $X\subseteq \Pp^r$ be a complete intersection of degrees $\boldsymbol{d}=(d_1,\dots,d_k)$ with $d_1\leqslant d_2\leqslant \dots\leqslant d_k$. We write 
\[\boldsymbol{d}=(d_1,\dots,d_k)=(\underbrace{d_1',\dots,d_1'}_{k_1 \text{ times }},\underbrace{d_2',\dots,d_2'}_{k_2 \text{ times }},\dots,\underbrace{d_\ell',\dots,d_\ell'}_{k_\ell \text{ times }}).\] We replace the moduli space of degree $d$ hypersurfaces with the moduli space of type $\boldsymbol{d}$ complete intersections $\mathcal{M}^{\GL_{r+1}}(\boldsymbol{d})$, which is constructed through a tower of Grassmann bundles and open immersions: 
\begin{center}
    \begin{tikzcd}
{\mathrm{Gr}(k_\ell,\mathcal{E}_{\ell})} \arrow[r] & {}_{\vdots} \arrow[d]    &                                            &                                            \\
              & U_2 \arrow[r, hook] & {\mathrm{Gr}(k_2,\mathcal{E}_2)} \arrow[d] &                                            \\
                  &                     & U_1 \arrow[r, hook]                        & {\mathrm{Gr}(k_1,\mathcal{E}_1)} \arrow[d] \\
        &                     &                                            & \Spec k                                   
\end{tikzcd}
\end{center}
Then, 
\[\mathcal{M}^{\GL_{r+1}}(\boldsymbol{d})=[\mathrm{Gr}(k_{\ell},\mathcal{E}_{\ell})/\GL_{r+1}].\]
Here, the bundle $\mathcal{E}_{i}$ is constructed such that the fiber over a $k_{i-1}$-dimensional subspace $\langle f_1,\dots,f_{k_{i-1}}\rangle\subset (\mathcal{E}_{i-1})_x$ is degree $d_i'$ polynomials modulo degree $(d_i'-d_{j}')$ multiples of the equations chosen at the preceding stages $j<i.$ For a detailed construction of the moduli space of complete intersections, see \cite[Section 1]{dilorenzo2022intersectiontheorymodulismooth}. We extend the universal $\mu$-incident line  $\Sigma_{\mu}$ to the moduli space of complete intersections. Proposition \ref{vanishing locus} extends in the following way:
\begin{thm}\label{complete intersections}
     Let $\mathcal{T}_i$ denote the pullback of the tautological subbundle of $[\mathrm{Gr}(k_i,\mathcal{E}_i)/\GL_{r+1}]$ to $\mathcal{M}^{\GL_{r+1}}(\boldsymbol{d}),$ and let $\mathcal{L}_i=\mathcal{O}_{[\Pp\mathcal{S}/\GL_{r+1}]}(d_i').$ The class of $\widetilde{\Sigma}_{\mu}$ is given by the formula
    \[[\widetilde{\Sigma}_{\mu}]=\prod_{i=1}^\ell c_{\mathrm{top}}(\mathcal{T}_{i}^\vee\boxtimes \mathcal{F}_{\mu}(\mathcal{L}_{i})).\]
\end{thm}
\begin{proof}[Proof Sketch]
We iteratively consider the maps 
$\pi_1^*\mathcal{T}_i\to \pi_2^*\mathcal{F}_{\mu}(\mathcal{L}_i)$
which are well defined when restricted to the vanishing locus $\mathcal{Z}_{i-1}$ of the previous map. Let $\alpha:\mathcal{Z}_{\ell-1}\to 
(\Pp\mathcal{S})^m\times_{\mathrm{BGL}_{r+1}}\mathcal{M}^{\GL_{r+1}}(\boldsymbol{d})$.
Inductively, we have that when $\widetilde{\Sigma}_{\mu}$ has expected codimension,
\[[\widetilde{\Sigma}_{\mu}]=\alpha_*\alpha^*c_{\mathrm{top}}(\mathcal{T}_{\ell}^\vee\boxtimes \mathcal{F}_{\mu}(\mathcal{L}_{\ell}))=[\mathcal{Z}_{\ell-1}]\cdot c_{\mathrm{top}}(\mathcal{T}_{\ell}^\vee\boxtimes \mathcal{F}_{\mu}(\mathcal{L}_{\ell}))=\prod_{i=1}^\ell c_{\mathrm{top}}(\mathcal{T}_{i}^\vee\boxtimes \mathcal{F}_{\mu}(\mathcal{L}_{i})).\]
Since $c(\mu)\leqslant d_j'+1$ for $1\leqslant j\leqslant \ell,$ an argument similar to Proposition \ref{vanishing locus} shows that indeed, $\widetilde{\Sigma}_{\mu}$ achieves expected codimension 
\[\codim \widetilde{\Sigma}_{\mu}=c(\mu)\sum_{j=1}^{\ell}k_j.\] 
\end{proof}
We note that the expected fiber dimension of $\pi_1|_{\Sigma_{\mu}}$ is therefore
\[d_{\mathrm{CI}}(\mu)=2(r-1)+m-c(\mu)\sum_{j=1}^{\ell}k_j.\]
The classes $[\Sigma_{\mu}],$ $[V_{\mu}],$ and $[S_{\mu}(X)]$ can be computed from $[\widetilde{\Sigma}_{\mu}]$ in the same way as in the hypersurface case. 
\begin{ex}
    Let $X$ be a degree $(2,3)$ complete intersection in $\Pp^3$, which is a genus $4$ canonically embedded curve. Let $\mu=(2,1)$, so a $\mu$-incident line is a tangent line meeting the curve at another point. We have that $d_{\mathrm{CI}}(\mu)=0$. By Theorem \ref{complete intersections}, 
    \begin{align*}
          [{\Sigma}_{\mu}]&=\sum_{j=0}^3\sum_{i=0}^3\gamma_{1}^{3-i}\gamma_2^{3-j}g_{\mu}^{i,j},
    \end{align*}
    and 
   \begin{align*}
       [S_{\mu}(X)]&=g_{\mu}^{3,3}|_{c_i=0}.
   \end{align*}
By Proposition \ref{collorary}, we have
\begin{align*}
    f_{\mu}^{3,3}&=\kappa_{-1,1,1}(\kappa_{-1,2,2}+\kappa_{0,2,1}+\kappa_{0,1,2}+\kappa_{1,1,1})\\&-2\kappa_{-1,3,2}-2\kappa_{-1,2,3}-2\kappa_{0,3,1}-8\kappa_{0,2,2}\\&-2\kappa_{0,1,3}-6\kappa_{1,2,1}-6\kappa_{1,1,2}-4\kappa_{2,1,1}.
\end{align*}
Applying the pullback formula 
\[\alpha^*(\kappa_{i,j,k})=f_*((-2\zeta+\sigma_1)^{i+1}(d_1\zeta)^j(d_2\zeta)^k),\]
we obtain 
\[[S_{\mu}(X)]=24\sigma_{2,2}.\]
Indeed, since $C=Q\cap S\subseteq \Pp^3$ is a complete intersection of a quadric and a cubic, by B\'ezout's theorem, any $\mu$-incident line is contained in the quadric. Since $Q\cong \Pp^1\times \Pp^1,$ every line $L\subseteq Q$ belongs to one of its rulings. The $\mu$-incident lines are precisely the lines of ramified fibers of one of the projection maps $f_{i}:C\to \Pp^1.$ Under the identification $Q\cong \Pp^1\times \Pp^1,$ $C\subseteq \Pp^1\times \Pp^1$ has bidegree $(3,3)$, so $\deg f_i=3.$
By Riemann-Hurwitz,
\[2g_C-2=(\deg f_i)\left(2g_{\Pp_1}-2\right)+\deg R_i,\]
so $\deg R_i=12$. Adding the ramification divisors from projecting onto both rulings, we compute that there are $24=\deg S_{\mu}(X)$ $\mu$-incident lines.
\end{ex}
\printbibliography
\end{document}